\documentclass[12pt]{article}

\usepackage{latexsym,amssymb,upref,amsmath,amsthm, amsfonts,authblk}
\usepackage{amssymb,amsmath,amsthm, calc, graphicx}
\usepackage{epsfig}
\usepackage{breqn}
\usepackage{footnpag}
\usepackage{rotating}
\usepackage{amsfonts}
\usepackage{setspace}
\usepackage{fullpage}
\usepackage{enumitem}
\usepackage{bbold}
\usepackage{comment}
\usepackage{pgf,tikz}
\usepackage{mathrsfs}
\usetikzlibrary{arrows}
\usepackage{yhmath}
\usepackage{hyperref}
\usepackage{authblk}

\newtheorem{thm}{Theorem}

\newtheorem{claim}{Claim}

\newtheorem{remark}[thm]{Remark}

\newtheorem{prop}[thm]{Proposition}

\newcommand{\thistheoremname}{}
\newtheorem*{genericthm*}{\thistheoremname}
\newenvironment{namedthm*}[1]
{\renewcommand{\thistheoremname}{#1}%
	\begin{genericthm*}}
	{\end{genericthm*}}

\newcommand{\abs}[1]{\left\lvert{#1}\right\rvert}

\title{Tur\'an number of an induced complete bipartite graph plus an odd cycle}

\author
{
Beka Ergemlidze
\thanks{ Department of Mathematics, Central European University, Budapest.
		E-mail: \texttt{beka.ergemlidze@gmail.com}} \qquad 
Ervin Gy\H{o}ri \thanks{R\'enyi Institute, Hungarian Academy of Sciences and 
	Department of Mathematics, Central European University, Budapest. E-mail: \texttt{gyori.ervin@renyi.mta.hu}} \qquad 
Abhishek Methuku \thanks{Department of Mathematics, Central European University, Budapest. E-mail: \texttt{abhishekmethuku@gmail.com}} 
}

\begin{document}

\maketitle

\begin{abstract}
 Let $k \ge 2$ be an integer. We show that if $s = 2$ and $t \ge 2$, or $s = t = 3$, then the maximum possible number of edges in a $C_{2k+1}$-free graph containing no induced copy of $K_{s,t}$ is asymptotically equal to $(t - s + 1)^{1/s}(n/2)^{2-1/s}$ except when $k = s =  t = 2$.
 
 This strengthens a result of Allen, Keevash, Sudakov and Verstra\"{e}te \cite{Allen_Keevash} and answers a question of Loh, Tait and Timmons \cite{Loh_Tait_Timmons}.
\end{abstract}

\section{Introduction}

Let $\mathcal F$ be a family of graphs. A graph is called $\mathcal F$-free if it does not contain any member of $\mathcal F$ as a subgraph. The  \emph{Tur\'an number} of $\mathcal F$ is the maximum number of edges in an $\mathcal F$-free graph on $n$ vertices and is denoted by $ex(n, \mathcal F)$.

The classical theorem of K\H{o}v\'ari, S\'os and Tur\'an \cite{KST1954} concerning the Tur\'an number of complete bipartite graphs states that $ex(n, K_{s,t}) \le \frac{1}{2} (t-1)^{1/s} \cdot n^{2-1/s} + O(n)$, where $s \le t$. Koll\'ar, R\'onyai and Szab\'o \cite{KRSZ1996} provided a lower bound matching the order of magnitude when $t > s!$, and later Alon, R\'onyai, and Szab\'o \cite{Alon_Ronyai_Szabo} provided a matching lower bound when $t > (s-1)!$. 

Let $ex(a, b, \mathcal F)$ denote the maximum number of edges in an $a$ by $b$ bipartite $\mathcal F$-free graph and let $ex_{bip}(n, \mathcal F)$ denote the maximum number of edges in an $n$-vertex bipartite $\mathcal F$-free graph.  F\"uredi \cite{F1996,FurediZarank} showed that if $a \le b$, then $ex(a, b, K_{s,t}) \le  (t - s + 1)^{1/s}ab^{1-1/s} + sa + sb^{2-2/s}$ for all $a \ge s$, $b \ge t$, $t \ge s \ge 2$. This easily implies the following.
\begin{equation}
\label{bipartiteturan}
ex_{bip}(n, K_{s,t}) \le (t - s + 1)^{1/s} (n/2)^{2-1/s} (1+o(1)).
\end{equation}


In the cases $s = 2$ and $s = t = 3$ asymptotically sharp values are known: F\"uredi \cite{F1996} determined the asymptotics for the Tur\'an number of $K_{2,t}$ by showing that for any fixed $t \ge 2$,  we have  $ex(n, K_{2,t}) = \sqrt{t-1} n^{3/2}/2 + O(n^{4/3}).$ By an example of Brown \cite{Brown} for the lower bound and a theorem of  F\"uredi \cite{FurediZarank} for the upper bound, it is known that $ex(n,K_{3,3})= n^{5/3}/2 +O(n^{5/3-c})$ for some $c>0$. 

Erd\H{o}s and Simonovits \cite{Erd_Simonovits} conjectured that given any family $\mathcal F$ of graphs, there exists $k \ge 1$ such that $ex(n, \mathcal F \cup \{C_3, C_5, \ldots, C_{2k+1}\})$ is equal to $ex_{bip}(n, \mathcal F)$ asymptotically. Allen, Keevash, Sudakov and Verstra\"{e}te \cite{Allen_Keevash} proved this conjecture for complete bipartite graphs $K_{2,t}$ and $K_{3,3}$ in a stronger form. 


\begin{thm}[Allen, Keevash, Sudakov and Verstra\"{e}te \cite{Allen_Keevash}]
\label{noninduced}
Let $k \ge 2$ be an integer. If $s = 2$ and $t \ge 2$, or $s = t = 3$, then $$ex(n, \{ C_{2k+1}, K_{s,t}\}) = (t - s + 1)^{1/s}(n/2)^{2-1/s} (1+o(1)).$$
\end{thm}


In fact, they proved a more general result for all \emph{smooth} families. Moreover, they also showed that any extremal $\mathcal F$-free $C_{2k+1}$-free graph is near-bipartite, which means a negligible number of edges may be deleted from it to make it bipartite. We note that the case $s = t = k = 2$ was solved earlier by Erd\H{o}s and Simonovits \cite{Erd_Simonovits}. 

\vspace{2mm}

In the rest of the paper we use the following asymptotic notation. Given two functions $f(n)$ and $g(n)$, we write $f(n) \sim g(n)$ if $\lim_{n \to \infty} \frac{f(n)}{g(n)} \to 1$. Otherwise, we write $f(n) \not \sim g(n)$.

\subsection{Induced Tur\'an numbers}

Given a graph $F$, let $F\text{-ind}$ denote an induced copy of $F$. Loh, Tait and Timmons \cite{Loh_Tait_Timmons} introduced the problem of simultaneously forbidding an induced copy of a graph and a (not necessarily induced) copy of another graph. Let $ex(n, H, F\text{-ind})$ denote the maximum possible number of edges in an $n$-vertex graph containing no induced copy of $F$ and no copy of $H$ as a subgraph. 

The question of determining $ex(n, H, F\text{-ind})$ is related to the well-studied areas of Ramsey-Tur\'an Theory and Erd\H{o}s-Hajnal Conjecture. The Ramsey-Tur\'an number $RT(n,H,m)$, is defined as the maximum number of edges that an $n$-vertex graph with independence number less than $m$ may have without containing $H$ as a subgraph. If we let $F$ be an independent set of order $m$ then $ex(n, H, F\text{-ind}) = RT(n,H,m)$. So the problem of determining $ex(n, H, F\text{-ind})$ is a generalization of the Ramsey-Tur\'an problem. We refer the reader to \cite{Loh_Tait_Timmons} for a detailed discussion on this general problem and its relation to other areas.

Among other interesting things, Loh, Tait and Timmons showed the following.

\begin{thm}[Loh, Tait, Timmons \cite{Loh_Tait_Timmons}]
	\label{LMT}
For any integers $k \ge 2$ and $t \ge 2$ there is a constant $\beta_k$, depending only on $k$, such that
$$ex(n, \{C_{2k+1}, K_{2,t}\text{-ind}\}) \le (\alpha(k, t)^{1/2} + 1)^{1/2} \frac{n^{3/2}}{2} + \beta_k n^{1+1/2k}$$
where $\alpha(k,t)=(2k-2)(t-1)((2k-2)(t-1)-1)$.
\end{thm}

They asked whether Theorem \ref{LMT} determines the correct growth rate in $k$, and showed that it determines the correct growth rate in $n$ and $t$. 

\vspace{2mm}

In this paper we answer their question by showing that $ex(n, \{C_{2k+1}, K_{2,t}\text{-ind}\})$ does not depend on $k$ asymptotically,  thus improving the upper bound in Theorem \ref{LMT} significantly. Our main result determines $ex(n, \{C_{2k+1}, K_{2,t}\text{-ind}\})$ asymptotically in all the cases except when $k = t = 2$, and is stated below.

\begin{thm}
	\label{mainthm}
For any integers $k \ge 2$, $t \ge 2$ where $(k, t) \not = (2, 2)$, we have
$$ex(n, \{C_{2k+1}, K_{2,t}\text{-ind}\}) = \sqrt{t-1} \cdot \left(\frac{n}{2}\right)^{3/2} (1+o(1)).$$
\end{thm}

Note that this shows $ex(n, \{C_{2k+1}, K_{2,t}\text{-ind}\}) \sim ex(n, \{ C_{2k+1}, K_{2,t}\})$ for all $k \ge 2$, $t \ge 2$ except in the case $k = t = 2$, which is studied in the next theorem. Therefore, Theorem \ref{mainthm} is a strengthening of Theorem \ref{noninduced} in the case $s = 2$ (except in the case $k = t = 2$); thus our proof of Theorem \ref{mainthm} provides a new proof of Theorem \ref{noninduced} in this case.

\vspace{2mm}

To prove the lower bound in Theorem \ref{mainthm} we use the following construction.

\textbf{Construction of an induced-$K_{2,t}$-free and $C_{2k+1}$-free graph: }
Consider a bipartite $K_{2,t}$-free graph $G$ with $n/2$ vertices in each color class and containing $\sqrt{t-1} \cdot (n/2)^{3/2}  + O(n^{4/3})$ edges. The existence of such a graph is shown by F\"uredi in \cite{F1996}. Clearly, $G$ contains no $C_{2k+1}$ as it is bipartite and because it contains no copy of $K_{2,t}$, of course, it contains no induced copy of $K_{2,t}$ as well. 

\vspace{2mm}

In the case $k = t =2$, we give the following improvement on the upper bound in Theorem \ref{LMT}.

\begin{thm}
\label{C_4C_5}
$$\frac{2}{3\sqrt{3}} n^{3/2}(1+o(1)) \le ex(n, \{C_5, K_{2,2}\text{-ind}\}) \le \frac{n^{3/2}}{2}(1+o(1)).$$
\end{thm}

To prove the lower bound, just as in \cite{Loh_Tait_Timmons}, we use the following example of Bollob\'as and Gy\H ori \cite{BolGy}. 

\textbf{Construction of an induced-$K_{2,2}$-free and $C_5$-free graph: }
Take a $C_4$-free bipartite graph $G_0$ on $n/3 + n/3$ vertices with about $(n/3)^{3/2}$ edges and double each vertex in one of the color classes and add an edge joining the old and the new copy to produce a graph $G$. It is easy to check that $G$ contains no $C_5$ and no induced copy of $C_4$. Moreover, $G$ contains approximately twice as many edges as $G_0$. 

\subsection{Organization of the paper}

We divide our proof of Theorem \ref{mainthm} into two cases: $k \ge 3$ and $k = 2$. The reason for doing so is explained in Remark \ref{split}. 

\vspace{2mm}

In Section \ref{notC_5}, we prove Theorem \ref{mainthm} in the case $k \ge 3$ and in Section \ref{C_5} we prove Theorem \ref{mainthm} in the case $k = 2$, along with Theorem \ref{C_4C_5}.  

\vspace{2mm}

Using Theorem \ref{mainthm}, Theorem \ref{noninduced} and Theorem \ref{C_4C_5}, we will show in the next section that for any given $k \ge 2$, $ex(n, \{C_{2k+1}, K_{s,t}\text{-ind}\})$ and $ex(n, \{C_{2k+1}, K_{s,t}\})$ are asymptotically equal whenever $s = 2$ and $t \ge 2$, or $s = t = 3$ except in the case $k = s = t = 2$; thus providing a strengthening of Theorem \ref{noninduced}.

\vspace{2mm}

In Section \ref{notation}, we introduce some notation used in our proofs as well as the Blakley-Roy inequality.

\section{$ex(n, \{C_{2k+1}, K_{s,t}\text{-ind}\})$ versus $ex(n, \{C_{2k+1}, K_{s,t}\})$}
\label{comparison}
First let us consider the case $s = 2$. As already noted, by comparing Theorem \ref{noninduced} and our Theorem \ref{mainthm}, it is easy to see that for all $k \ge 2$, $t \ge 2$ except when $k = t = 2$, we have
\begin{equation}
\label{eq:sis2}
ex(n, \{C_{2k+1}, K_{2,t}\text{-ind}\}) \sim ex(n, \{ C_{2k+1}, K_{2,t}\})
\end{equation}

Now consider the case $s= t = 3$. In Proposition \ref{indtononind} (in Section \ref{notC_5}) we prove that for any $k, s, t \ge 2$,
\begin{equation}
\label{eq:one}
ex(n, \{C_{2k+1}, K_{s,t}\text{-ind}\}) \le ex(n, \{C_3, C_{2k+1}, K_{s,t}\}) + 3c_k n^{1+1/k},
\end{equation}
where $c_k$ is a constant depending only on $k$.

It follows from Theorem \ref{noninduced} that for all $k \ge 2$, we have $ex(n, \{C_3, C_{2k+1}, K_{3,3}\}) \le (n/2)^{5/3} (1+o(1))$. Combining this with \eqref{eq:one} we get, $$ex(n, \{C_{2k+1}, K_{3,3}\text{-ind}\}) \le (n/2)^{5/3} (1+o(1))+ 3c_k n^{1+1/k}.$$
Now note that $3c_k n^{1+1/k} = o(n^{5/3})$ for all $k \ge 2$. Therefore,
$ex(n, \{C_{2k+1}, K_{3,3}\text{-ind}\}) \le (n/2)^{5/3} (1+o(1)) = ex(n, \{ C_{2k+1}, K_{3,3}\})$. This implies,
\begin{equation}
\label{eq:stis3}
ex(n, \{C_{2k+1}, K_{3,3}\text{-ind}\}) \sim ex(n, \{ C_{2k+1}, K_{3,3}\}).
\end{equation}

Therefore, \eqref{eq:sis2} and \eqref{eq:stis3} imply the following strengthening of Theorem \ref{noninduced} in all but one special case.

\begin{thm}
	Let $k \ge 2$ be an integer. If $s = 2$ and $t \ge 2$, or $s = t = 3$, then $$ex(n, \{C_{2k+1}, K_{s,t}\text{-ind}\}) = (t - s + 1)^{1/s}(n/2)^{2-1/s} (1+o(1)),$$
	except when $k = s =  t = 2$.
\end{thm}

Surprisingly, the case $k = s = t = 2$ is quite different. As noted before, in this case it is known by a theorem of Erd\H{o}s and Simonovits \cite{Erd_Simonovits} (or by Theorem \ref{noninduced}) that $ex(n, \{ C_{5}, K_{2,2}\}) = (n/2)^{3/2}(1+o(1))$ where as  $ex(n, \{C_5, K_{2,2}\text{-ind}\}) \ge (2n^{3/2}/{3\sqrt{3}}) (1+o(1))$ by the lower bound in Theorem \ref{C_4C_5} (as observed by Loh, Tait and Timmons in \cite{Loh_Tait_Timmons}). Therefore, in this very special case $$ex(n, \{C_5, K_{2,2}\text{-ind}\}) \not \sim ex(n, \{ C_{5}, K_{2,2}\}).$$

\section{Notation and the Blakley-Roy inequality}
\label{notation}

Let $G = (V(G), E(G))$ be a graph. For convenience, sometimes we refer to $E(G)$ by $G$. So sometimes we write $\abs{G}$ to denote the number of edges in $G$. Given a set $S \subseteq V(G)$, the subgraph of $G$ induced by $S$ is denoted $G[S]$. 

Given a graph and a vertex $v$ in it, the first neighborhood of $v$,  $N_1(v)$ is the set of vertices adjacent to $v$ and for $i \ge 2$, let $N_i(v)$ denote the set of vertices at distance exactly $i$ from $v$. Notice that for $i \not = j$, $N_i(v) \cap N_j(v) = \emptyset$.


A \emph{3-walk} is a sequence $v_0e_0v_1e_1v_2e_2v_3$ of vertices and edges such that $e_i = v_iv_{i+1}$ for $0 \le i \le 2$. For convenience we simply denote such a 3-walk by $v_0v_1v_2v_3$. The vertices $v_0$ and $v_3$ are called the first and last vertices of the 3-walk respectively and the edges $e_0$ and $e_2$ are called the first and last edges respectively. We say the $3$-walk starts with the edge $e_0$ and ends with the edge $e_2$. We refer to $v_0, v_1, v_2, v_3$ as first, second, third and fourth vertices of the walk respectively. Note that the 3-walks $v_0v_1v_2v_3$ and $v_3v_2v_1v_0$ are generally considered different. Also note that edges can be repeated in a 3-walk. A \emph{3-path} is a 3-walk with no repeated vertices or edges. 

Blakley and Roy \cite{Blakley_Roy} proved a matrix version of H\"older's inequality, which implies that any graph $G$ with average degree $d$ has at least $nd^3$ 3-walks.

\section{Proof of Theorem \ref{mainthm} and Theorem \ref{C_4C_5}}
\label{upperbound}

Let $G$ be a $C_{2k+1}$-free graph containing no induced copy of $K_{2,t}$.
In Section \ref{notC_5}, we consider the case $k \ge 3$ and in Section \ref{C_5} we consider the case $k = 2$.

\subsection{When $k \ge 3$}
\label{notC_5}
In this section we prove Theorem \ref{mainthm} in the case $k \ge 3$. We only need the following proposition for $s = 2$, but we will prove it in a more general form because we need it in Section \ref{comparison}.

\vspace{2mm}

\begin{prop}
	\label{indtononind}
For any $k, s, t \ge 2$ we have $$ex(n, \{C_{2k+1}, K_{s,t}\text{-ind}\}) \le ex(n, \{C_3, C_{2k+1}, K_{s,t}\}) + 3c_k n^{1+1/k},$$ where $c_k$ is a constant depending only on $k$.
\end{prop}

\begin{proof}
Let $G$ be a $C_{2k+1}$-free graph containing no induced copy of $K_{s,t}$. Let $G_{\Delta}$ be the subgraph of $G$ consisting of the edges which are contained in the triangles of $G$. Let $G \setminus G_{\Delta}$ be the graph obtained after deleting all the edges of $G_{\Delta}$ from $G$. Of course, $G \setminus G_{\Delta}$ is triangle free, and the number of edges in $G_{\Delta}$ is at most three times the number of triangles in $G$.

Gy\H{o}ri and Li \cite{Gyori_Li} showed that in a $C_{2k+1}$-free graph there are at most $c_k n^{1+1/k}$ triangles where $c_k$ is a constant depending only on $k$. This implies that $\abs{E(G_{\Delta})} \le 3c_k n^{1+1/k}$. 


\begin{claim}
	\label{kst_removed_triangles_claim}
	$G \setminus G_{\Delta}$ is $K_{s,t}$-free.
\end{claim}

\begin{proof}
	Assume for a contradiction that there is a (not necessarily induced) copy of $K_{s,t}$ in $G \setminus G_{\Delta}$, and let $A$, $B$ be its color classes. Then since $G$ contains no induced copy of $K_{s,t}$, there must be an edge $xy$ of $G$ which is contained in $A$ or $B$. In either case, it is easy to see that $xy$ and some two edges of the $K_{s,t}$ form a triangle. However, this means that these two edges of $K_{s,t}$ are contained in $G_{\Delta}$ by definition, contradicting the assumption that this $K_{s,t}$ is contained in $G \setminus G_{\Delta}$.  Therefore, the claim follows. 
\end{proof}

By Claim \ref{kst_removed_triangles_claim}, $G \setminus G_{\Delta}$ is $K_{s,t}$-free. Moreover, as $G \setminus G_{\Delta}$ is a triangle free subgraph of $G$, we have $\abs{E(G \setminus G_{\Delta})} \le ex(n, \{C_3, C_{2k+1}, K_{2,t}\})$.

Therefore, $\abs{E(G)} = \abs{E(G_{\Delta})}  + \abs{E(G \setminus G_{\Delta})} \le 3c_k n^{1+1/k} + ex(n, \{C_3, C_{2k+1}, K_{s,t}\})$, proving Proposition \ref{indtononind}.
\end{proof}

Proposition \ref{oddcycle_triangle_bipartite} will show that $ex(n, \{C_3, C_{2k+1}, K_{2,t}\}) \le \sqrt{t-1} \cdot \left(n/2\right)^{3/2} (1+o(1))$ for all $k, t \ge 2$. Combining this with Proposition \ref{indtononind} for $s =2$, we get,
\begin{equation}
\label{eq:sumofgdeltaandgminusgdelta}
ex(n, \{C_{2k+1}, K_{2,t}\text{-ind}\}) \le 3c_k n^{1+1/k} + \sqrt{t-1} \cdot \left(\frac{n}{2}\right)^{3/2} (1+o(1)).
\end{equation}

Since $k \ge 3$,  clearly $3c_k n^{1+1/k} = o(n^{3/2})$, completing the proof of Theorem \ref{mainthm} in the case $k \ge 3$.

\begin{remark}
	\label{split}
Note that when $k = 2$, the number of edges in $G_{\Delta}$ can be as large as $\Theta(n^{3/2})$. For example, observe that this is the case in the Bollob\'as-Gy\H ori construction which is stated after Theorem \ref{C_4C_5} (note that every edge in the construction is contained in a triangle). So using Proposition \ref{indtononind}, we cannot get a better upper bound than $(3c_2 + \sqrt{t-1} \cdot \left(\frac{1}{2}\right)^{3/2}) n^{3/2}(1+o(1))$ where $c_2 > 0$ is a constant. Therefore, the current approach does not work in the case $k = 2$; we will give a different proof for this case in Section \ref{C_5}.
\end{remark}

It only remains to prove Proposition \ref{oddcycle_triangle_bipartite}. 

\begin{prop}
	\label{oddcycle_triangle_bipartite}
For all integers $k \ge 2, t \ge 2$, we have
$$ex(n, \{C_3, C_{2k+1}, K_{2,t}\}) \le \sqrt{t-1} \cdot \left(\frac{n}{2}\right)^{3/2} (1+o(1)).$$
\end{prop}

\begin{remark}
	Note that Proposition \ref{oddcycle_triangle_bipartite} follows from Theorem \ref{noninduced}. However, as our proof is simple we present it below for completeness. 
	
	Also note that the proof given in this section provides a different proof of Theorem \ref{noninduced} when $k \ge 3$ and $s = 2$. A key new idea is to remove all the edges contained in triangles first (recall that they are negligible in this case), and this helps avoid some of the technicalities that would otherwise arise in the proof. For example, after destroying all of the triangles, it is straightforward that for any vertex $v$ in the resulting graph, $N_1(v)$ does not induce any edges and (as will be shown in the proof of Proposition \ref{oddcycle_triangle_bipartite}) it is also easier to argue that $N_2(v)$ does not induce many edges.
\end{remark}

\begin{proof}[Proof of Proposition \ref{oddcycle_triangle_bipartite}]
Let $d$ be the average degree of a graph $H =(V(H), E(H))$ containing no triangle, $C_{2k+1}$ or $K_{2,t}$. Our aim is to upper bound $d$. If a vertex has degree smaller than $d/2$, then we can delete the vertex and the edges incident on it without decreasing the average degree. 
and it is easy to see that if the desired upper bound on the average degree holds for the new graph, then it holds for the original graph as well. Therefore, we may assume that $H$ has minimum degree at least $d/2$.

Suppose there is a vertex $u$ in $H$ with degree more than $4d$. Of course, $\abs{N_1(u)} > 4d$. Since the minimum degree in $H$ is at least $d/2$, there are at least $(d/2 - 1) \abs{N_1(u)}$ edges between $N_1(u)$ and $N_2(u)$ as there are no edges contained in $N_1(u)$ because $H$ is triangle free. On the other hand, if a vertex $w \in N_2(u)$ is adjacent to $t$ vertices in $N_1(u)$, then $u, w$ and these $t$ vertices form a $K_{2,t}$, a contradiction. So there are at most $\abs{N_2(u)}(t-1)$ edges between $N_1(u)$ and $N_2(u)$. Combining, we get, $\abs{N_2(u)}(t-1) \ge (d/2 - 1) \abs{N_1(u)} > (d/2 - 1) 4d = 2d^2-4d$. Now since, $n \ge \abs{N_1(u)}+\abs{N_2(u)}$, we get $n > 4d + (2d^2-4d)/(t-1) = 2d^2/(t-1)$. Therefore, $d < \sqrt{(t-1)}\sqrt{n/2}$ and the bound in Proposition \ref{oddcycle_triangle_bipartite} holds because $\abs{E(H)} = dn/2$. 

So from now on we can assume that the maximum degree $d_{max}$ in $H$ is at most $4d$. First let us show that $N_2(v)$ doesn't induce many edges. 

\begin{claim}
\label{N_2_size}
The number of edges induced by $N_2(v)$ is at most $(2k-4)16d^2$.
\end{claim}
\begin{proof}
For each $q \in N_1(v)$, let $S_q$ be the set of neighbors of $q$ in $N_2(v)$. Of course the sets $\{S_q \mid q \in N_1(v) \}$ cover all the vertices of $N_2(v)$.  So we can choose sets $S'_q \subset S_q$ such that $\{S'_q \mid q \in N_1(v) \}$ partition $N_2(v)$. Note that the sets $S'_q$ do not induce any edges as $H$ is triangle-free, so the edges induced by $N_2(v)$ are between the sets $S'_q$, $q \in N_1(v)$. Color each set $S'_q$ red or blue with probability $1/2$. Each vertex in $N_2(v)$ is colored by the color of the set it is contained in. It is easy to see that there exists a coloring of the sets $S'_q$ such that at least half of all the edges induced by $N_2(v)$ are not monochromatic. (Indeed, the probability that an edge of $N_2(v)$ is not monochromatic is $1/2$.)  If there is a path of length $2k-3$ in the graph $B$ consisting of these non-monochromatic edges, then since $2k-3$ is odd, the end vertices $y_1, y_2$ of the path are of different colors, so there exist distinct vertices $x_1, x_2 \in N_1(v) $ such that $x_1y_1, x_2y_2 \in E(H)$. However, then $vx_1, vx_2, x_1y_1, x_2y_2$ and this path of length $2k-3$ form a $2k+1$ cycle in $H$, a contradiction. Therefore, using Erd\H{o}s-Gallai theorem \cite{Erd_Gallai} and the fact that $B$ contains at least half of the edges induced by $N_2(v)$, we get 
\begin{equation*}
\abs{H[N_2(v)]} \le \frac{2k-3-1}{2} \abs{N_2(v)} \cdot 2 = (2k-4)\abs{N_2(v)} \le (2k-4) d_{max}^2 \le (2k-4)16d^2,
\end{equation*}
proving the claim.
\end{proof}

By the Blakley-Roy inequality, there are at least $nd^3$ 3-walks in $H$, so there exists a vertex $v$ which is the first vertex of at least $d^3$ 3-walks. A 3-walk of the form $vv_1v_2v_3$ where $v_i \in N_i(v)$ for $1 \le i \le 3$ is called as \emph{good}. If a 3-walk is not good but has $v$ as its first vertex, then either $v_2 = v$ or $v_2 \in N_2(v), v_3 \in N_1(v)$  or $v_2 \in N_2(v), v_3 \in N_2(v)$. (Note that here we used that $N_1(v)$ doesn't contain any edges - as $H$ is triangle-free.) Below we show that the number of 3-walks starting from $v$ that are not good are very few.

Number of 3-walks starting from $v$ where $v_2 = v$ is at most $d_{max}^2 \le 16d^2$. Indeed, there are at most $d_{max}$ choices for $v_1$ and $v_3$ since they are both adjacent to $v$. Now we estimate the number of  3-walks starting from $v$ where $v_2 \in N_2(v)$, $v_3 \in N_1(v)$. Any given $v_2 \in N_2(v)$ has at most $t-1$ neighbors in $N_1(v)$ for otherwise we can find a copy of $K_{2,t}$ in $H$, a contradiction. So the number of such 3-walks is at most $d_{max}^2(t-1) \le 16(t-1)d^2$ because there are at most $d_{max}, d_{max}$ and $t-1$ choices for $v_1, v_2$ and $v_3$ respectively. Finally, we estimate the number of 3-walks where $v_2 \in N_2(v)$ and $v_3 \in N_2(v)$. So the edge $v_2v_3 \in H[N_2(v)]$. For a given edge $xy \in H[N_2(v)]$, either $v_2 = x, v_3 = y$ or $v_2 = y, v_3 = x$ and for fixed $v_2, v_3$, there are at most $t-1$ choices for $v_1 \in N_1(v)$. Therefore, the number of such 3-walks is at most $\abs{H[N_2(v)]}2(t-1)$. Now by Claim \ref{N_2_size}, this is at most $ (2k-4)16d^2\cdot2(t-1) = 64(k-2)(t-1)d^2$. Therefore, by summing these estimates, we get that the number of 3-walks starting from $v$ that are not good is at most $16d^2+ 16(t-1)d^2 + 64(k-2)(t-1)d^2 = 16d^2 + 16(4k-7)(t-1)d^2 \le 32(4k-7)(t-1)d^2.$ 

Thus the number of good 3-walks is at least $d^3 -32(4k-7)(t-1)d^2$. Let us consider the graph $B'$ consisting of the last edges of these good 3-walks. Observe that $B'$ is a $K_{2,t}$-free bipartite graph with color classes $N_2(v)$ and $N_3(v)$. It is easy to see that an edge of $B'$ belongs to at most $t-1$ good 3-walks, otherwise we can find a $K_{2,t}$. Therefore there are at least $(d^3 -32(4k-7)(t-1)d^2)/(t-1)$ edges in $B'$. It follows from \eqref{bipartiteturan} (or by a simple double counting of cherries)  that a bipartite $K_{2,t}$-free graph on $n$ vertices (or less) contains at most $\sqrt{t-1}(n/2)^{3/2} (1+o(1))$ edges, so $B'$ contains at most this many edges. Combining the two estimates, we get 
$$(d^3 -32(4k-7)(t-1)d^2)/(t-1) \le  \sqrt{t-1}(n/2)^{3/2} (1+o(1)).$$
Rearranging, we get $(d-c_{k,t})^3 \le  (t-1)^{3/2}(n/2)^{3/2} (1+o(1))$
where $c_{k,t} = \frac{32}{3}(4k-7)(t-1)$. Therefore, $d \le \sqrt{t-1}(n/2)^{1/2} (1+o(1))$ which implies that $$\abs{E(H)} = nd/2 \le \sqrt{t-1}(n/2)^{3/2}(1+o(1))$$ completing the proof of Proposition \ref{oddcycle_triangle_bipartite}, and concluding this subsection.
\end{proof}

\subsection{When $k = 2$}
\label{C_5}

Here we prove Theorem \ref{mainthm} in the case $k =2$ and Theorem \ref{C_4C_5} together. 

\vspace{3mm}

Let $d$ be the average degree of $G$. It suffices to show that $d \le \sqrt{(\max(3,t)-1)} \sqrt{n/2} (1+o(1))$. As usual we can assume that each vertex of $G$ has degree at least $d/2$, for otherwise we can delete it and the edges incident on it to obtain a new graph with average degree at least $d$.

\begin{claim}
\label{limited_cherries}
Any two non-adjacent vertices $u, v$ in $G$ have at most $\max(3,t)-1$ common neighbors. 
\end{claim}
\begin{proof}
Suppose for a contradiction that $u$ and $v$ have $\max(3,t)$ common neighbours and let $S$ denote the set of these common neighbours. Then since $u, v$ and vertices in $S$ cannot form an induced copy of $K_{2,t}$, there must be an edge $xy$ among the vertices in $S$. However, then $uxyvw$ is a five cycle in $G$ for some $w \in S \setminus \{u,v\}$, a contradiction.
\end{proof}

Now we will show that we can assume the maximum degree $d_{max}$ of $G$ is at most $6d$. Suppose that there is a vertex $v$ in $G$ with degree more than $6d$. Of course $\abs{N_1(v)} > 6d$. Since $G$ is $C_5$-free, there is no path on $4$ vertices in the subgraph of $G$ induced by $N_1(v)$. Therefore, the number of edges induced by $N_1(v)$ is at most $(4-2)/2 \cdot \abs{N_1(v)} =  \abs{N_1(v)}$ by Erd\H{o}s-Gallai theorem \cite{Erd_Gallai}. Since the minimum degree is at least $d/2$, the sum of degrees of the vertices in $N_1(v)$ is at least $d/2 \cdot \abs{N_1(v)}$. In this sum, the edges between $v$ and $N_1(v)$ are counted once, the edges induced by $N_1(v)$ are counted twice, so the number of edges between $N_1(v)$ and $N_2(v)$ is at least $d/2 \abs{N_1(v)} -\abs{N_1(v)} -2\abs{N_1(v)} = (d/2-3) \abs{N_1(v)}$. On the other hand, a vertex $u$ in $N_2(v)$ is adjacent to at most $\max(3,t)-1$ vertices in $N_1(v)$ by applying Claim \ref{limited_cherries} to $u$ and $v$. So there are at most $\abs{N_2(v)}(\max(3,t)-1)$ edges between $N_1(v)$ and $N_2(v)$. So combining, we get $ \abs{N_2(v)}(\max(3,t)-1) \ge (d/2-3) \abs{N_1(v)} > (d/2-3)6d = 3d^2-18d$. Since, $n \ge \abs{N_1(v)} + \abs{N_2(v)}$, we have that $n >6d + (3d^2-18d)/(\max(3,t)-1) \ge (3d^2-6d)/(\max(3,t)-1) \ge 2d^2/(\max(3,t)-1)$ whenever $d \ge 6$. Therefore, $d < \sqrt{(\max(3,t)-1)} \sqrt{n/2} (1+o(1))$, so $\abs{E(G)} = nd/2 < \sqrt{(\max(3,t)-1)} (n/2)^{3/2}(1+o(1))$, proving Theorem \ref{mainthm} in the case $k = 2$ and $t \ge 3$ and Theorem \ref{C_4C_5}. 

So from now on, we may assume $d_{max} \le 6d$. A 3-path is called \emph{good} if the second and fourth vertex in it are not adjacent.

\begin{claim}
	\label{good3paths_lower}
 There is an edge $xy$ in $G$ such that the number of good 3-paths starting with $xy$ is at least $2d^2 -84d$.
\end{claim}

\begin{proof}
By the Blakley-Roy inequality, there are at least $nd^3$ 3-walks in $G$, so there is an edge $xy$ that is the first edge of at least $nd^3/(\abs{E(G)} = nd^3/(nd/2) = 2d^2$ 3-walks. First let us show that most of these $2d^2$ 3-walks are 3-paths. Suppose $xyzw$ is a 3-walk that is not a 3-path. Then either $z = x$ or $w = y$ or $w = x$. In each of these cases, there are at most $d_{max}$ 3-walks, so in total there are at most $3d_{max} \le 18d$ such 3-walks. Similarly there are at most $18d$ 3-walks of the form $yxzw$ that are not 3-paths. Therefore, there are at least $2d^2 -36d$ 3-paths starting with the edge $xy$.

If a 3-path $xyzw$ or $yxzw$ is not good, then the edge $zw \in G[N_1(y)]$ or $zw \in G[N_1(x)]$ respectively. Moreover, given an edge $zw \in G[N_1(y)]$ or $zw \in G[N_1(x)]$ there are at most 4 paths starting with the edge $xy$ and containing $zw$ as its last edge, so the total number of 3-paths starting with the edge $xy$ that are not good is at most $4(\abs{G[N_1(x)]} + \abs{G[N_1(y)]})$. Since the neighborhood of a vertex doesn't contain a path on $4$ vertices, $$\abs{G[N_1(x)]} + \abs{G[N_1(y)]} \le (4-2)/2 \cdot \abs{N_1(x)} + (4-2)/2 \cdot \abs{N_1(y)} \le 2d_{max} \le 12d$$ by Erd\H{o}s-Gallai theorem \cite{Erd_Gallai}. So the number of good 3-paths starting with the edge $xy$ is at least $2d^2 -36d - 48d = 2d^2 -84d$.
\end{proof}

\begin{claim}
	\label{w}
For any vertex $w$ different from $x$ and $y$, at most $\max(3,t)-1$ good 3-paths start with the edge $xy$ and have $w$ as their last vertex. 
\end{claim}

\begin{proof}
Suppose for a contradiction that $\max(3,t)$ good 3-paths start with the edge $xy$ and have $w$ as their last vertex. 

Let us suppose that among these good 3-paths, there is one of the form $xyz_1w$ and another of the form $yxz_2w$. Now if $z_1 \not = z_2$, then $xyz_1wz_2$ forms a $C_5$ in $G$, a contradiction. Therefore, $z_1 = z_2$. Now since $\max(3,t) \ge 3$, there must be another good 3-path starting with the edge $xy$ and having $w$ as its last vertex - it is either of the form $xyz'w$ or of the form $yxz'w$ for some $z' \not = z_1$. In the first case, $xyz'wz_1$ is a $C_5$ and in the second case $yxz'wz_1$ is a $C_5$ in $G$, a contradiction. Therefore, the $\max(3,t)$ good 3-paths starting with the edge $xy$ and having $w$ as their last vertex are all either of the form $xyv_1w, xyv_2w, \ldots, xyv_{\max(3,t)}w$ or of the form $yxv_1w, yxv_2w, \ldots, yxv_{\max(3,t)}w$ where $v_1, v_2, \ldots, v_{\max(3,t)}$ are distinct vertices. However, in the first case, $y$ and $w$ are non-adjacent (as these are good 3-paths) and in the second case $x$ and $w$ are non-adjacent. Moreover, in both of these cases they have  $v_1, v_2, \ldots, v_{\max(3,t)}$ as their common neighbors, contradicting Claim \ref{limited_cherries}.
\end{proof}

It follows from Claim \ref{good3paths_lower} and Claim \ref{w} that there are at least $(2d^2 -84d)/(\max(3,t)-1)$ vertices in $G$. Therefore, $(2d^2 -84d)/(\max(3,t)-1) \le n$. Simplifying, we get $$(d-21)^2 \le \frac{n}{2}(\max(3,t)-1)+441.$$ So, $d \le \sqrt{(\max(3,t)-1)} \sqrt{n/2} (1+o(1)),$ implying that $$\abs{E(G)} = nd/2 \le \sqrt{(\max(3,t)-1)} (n/2)^{3/2}(1+o(1)),$$ completing the proof of Theorem \ref{mainthm} in the cases $k = 2$ and $t \ge 3$ and Theorem \ref{C_4C_5} (after putting $k = t = 2$).

\section*{Acknowledgements}
The research of the authors is partially supported by the National Research, Development and Innovation Office  NKFIH, grant K116769. We thank Casey Tompkins for bringing \cite{Loh_Tait_Timmons} to our attention.

%
%
%
%


\begin{thebibliography}{10}
		\bibitem{Allen_Keevash}
		P. Allen, P. Keevash, B. Sudakov, J. Verstra\"{e}te. \emph{Tur\'an numbers of bipartite graphs plus an odd cycle.} Journal of Combinatorial Theory, Series B. 2014 May 31;106:134-62.
		
		\bibitem{Alon_Ronyai_Szabo}
		N. Alon, L. R\'onyai, and T. Szab\'o. \emph{Norm-graphs: variations and applications.} J. Combin. Theory Ser. B 76 (1999), 280–290.
		
		\bibitem{Blakley_Roy}
		R. G. Blakley, and P. Roy. A H\"{o}lder type inequality for symmetric matrices with nonnegative entries. \emph{Proceedings of the American Mathematical Society} 16.6 (1965): 1244-1245.
		
		\bibitem{BolGy} B. Bollob\'as and E. Gy\H ori. \emph{Pentagons vs. triangles.} Discrete Mathematics \textbf{308.19} (2008) 4332-4336.
		
		\bibitem{Brown} W. G. Brown. \emph{On graphs that do not contain a Thomsen graph.} Canad. Math. Bull. 9 (1966), 281–285.
		
		\bibitem{Erd_Simonovits} P. Erd\H{o}s and M. Simonovits, \emph{Compactness results in extremal graph theory.} Combinatorica 2 (1982) 275–288.
		
		\bibitem{Erd_Gallai}
		P. Erd\H{o}s and T. Gallai. \emph{On maximal paths and circuits of graphs.} Acta Math. Acad. Sci. Hungar. 10 (1959), 337–356.
		
		\bibitem{Erdos_Ren_Sos}
		P. Erd\H{o}s, A. R\'enyi, and Vera T. S\'os. \emph{On a problem of graph theory} Stud Sci. Math. Hung. 1 (1966), 215–235.
		
		\bibitem{FurediZarank}
		Z. F\"uredi. \emph{An upper bound on Zarankiewicz problem.} Combin. Probab. Computing 5 (1996), no. 1, 29–33.
		
		\bibitem{F1996} Z. F\"uredi. \emph{New asymptotics for bipartite Tur\'an numbers.} Journal of Combinatorial Theory, Series A, \textbf{75}(1), 141-144, (1996).
		
		\bibitem{Gyori_Li} E. Gy\H{o}ri and H. Li. \emph{The maximum number of triangles in $C_{2k+1}$-free graphs.} Combinatorics, Probability and Computing, 21.1-2 (2012): 187.
		
		\bibitem{KRSZ1996} J. Koll\'ar, L. R\'onyai, T. Szab\'o. \emph{Norm-graphs and bipartite Tur\'an numbers.} \textit{Combinatorica}, \textbf{16}(3), 399--406, (1996).
		
		\bibitem{KST1954} T. K\H{o}v\'ari, V. S\'os, P. Tur\'an. \emph{On a problem of K. Zarankiewicz.} In Colloquium Mathematicae, Vol. 3, No. 1, pp. 50--57, (1954).
		
		\bibitem{Loh_Tait_Timmons} Po-Shen Loh, M. Tait and C. Timmons. Induced Tur\'an numbers. arXiv preprint arXiv:1610.06521 (2016).
	\end{thebibliography}
\end{document}